\documentclass[a4paper,12pt,leqno]{article}

\usepackage{lmodern}        
\usepackage[T1]{fontenc}    

\usepackage{url}
\usepackage[pagebackref]{hyperref}
\usepackage{amsfonts}
\usepackage{amsthm}
\usepackage{amsmath}
\usepackage{amscd}
\usepackage{amssymb}
\usepackage{mathtools}
\usepackage{tikz}
\usepackage{tikz-cd}
\usepackage{geometry}

\def\CC {{\mathbb C}}     
\def\PP {{\mathbb P}}     
\def\QQ {{\mathbb Q}}     
\def\RR {{\mathbb R}}     
\def\ZZ {{\mathbb Z}}     

\def\ring#1{\ifmmode \mathaccent'027 #1\else \rm\accent'027 #1\fi}

\def\mc {\mathcal}

\def\Det {\mathrm{Det}}

\mathchardef\mhyphen="2D

\newtheorem{theorem}{Theorem}[section]

\newtheorem{lemma}[theorem]{Lemma}
\newtheorem{prop}[theorem]{Proposition}
\newtheorem{coro}[theorem]{Corollary}
\newtheorem{remark}[theorem]{Remark}
\theoremstyle{definition}
\newtheorem{df}[theorem]{Definition}

\theoremstyle{plain}

\newtheorem{q}[theorem]{Question}

\numberwithin{equation}{section}

\newlength{\miniwidth}

\begin{document}

\title{An extension of the Siegel space of complex abelian varieties and conjectures on stability structures}
\author{Fabian Haiden}
\maketitle

\begin{abstract}
We study semi--algebraic domains associated with symplectic tori and conjecturally identified with spaces of stability conditions on the Fukaya categories of these tori.
Our motivation is to test which results from the theory of flat surfaces could hold for more general spaces of stability conditions.
The main results concern systolic bounds and volume of the moduli space.
\end{abstract}

\newpage

\setcounter{tocdepth}{2}
\tableofcontents

\section{Introduction}

The theory of quadratic differentials on Riemann surfaces has seen surprising connections with a number of different areas of mathematics. 
From a geometric point of view they give flat surfaces with conical singular points whose cone angles are integer multiples of $\pi$.
We refer to \cite{zorich} for a broad introduction.
A particularly fundamental result, due to Masur~\cite{masur82} and Veech~\cite{veech82}, is the finiteness of the volume of the moduli space of flat surfaces of given type, which allows methods from ergodic theory to come into play.

More recently, motivated partly by considerations in string theory, moduli spaces of flat surfaces were shown to be special instances of spaces of stability conditions, as defined by Bridgeland, in works by Bridgeland--Smith~\cite{bs}, and Kontsevich, Katzarkov, and the author~\cite{hkk}.
In the latter work \cite{hkk}, a triangulated category, the \textit{partially wrapped Fukaya category} $\mc F(S)$, is defined for surfaces $S$ with markings. 
It is then shown that a flat metric on $S$ induces a stability condition on $\mc F(S)$ whose stable objects correspond to (finite length) geodesics on $S$.
In this way one obtains an embedding of the space of flat metrics on $S$, $\mc M(S)$, into the space of stability conditions, $\mathrm{Stab}(\mc F(S))$, of $\mc F(S)$ as a union of connected components of the latter.
Moreover, all of the important features of $\mc M(S)$, such as the period map and the wall--and--chamber structure, can be defined intrinsically on $\mathrm{Stab}(\mc F(S))$ starting from the axiomatics of a stability condition.

Given these results it is natural to wonder how much of the theory of flat surfaces survives for spaces of stability conditions on more general categories, e.g. Fukaya categories, $\mc F(M)$, of higher--dimensional symplectic manifolds, $M$.
See also the discussion by Smith~\cite{smith_stabsymp}.
The case when $\dim_\RR M=6$ and $M$ is compact is of particular interest in view of the theory of categorical Donaldson--Thomas invariants~\cite{ks}.
An immediate difficulty is that constructing stability conditions and determining the geometry of $\mathrm{Stab}(\mathcal C)$ is a hard problem in higher dimensions, though progress has been made~\cite{bmt, mp, bms}.

In the case when $M$ is a torus with rational constant symplectic form, there is an explicit conjectural description of $\mathrm{Stab}(\mc F(M))$ due to Kontsevich~\cite{kontsevich_lille}. 
It involves the following generalization of the notion of a compatible complex structure on a symplectic vector space, $(V,\omega)$.
Suppose $\dim_\RR V=2n$ and let $\mc U(V)$ be the space of $\CC$--valued alternating $n$-forms on $V$ which are primitive (i.e. $\Omega\wedge\omega=0$) and whose restriction to any Lagrangian subspace $L\subset V$ is non-vanishing (as a top-degree form on $L$).
It turns out that $\mc U(V)$ has two connected components $\mc U^\pm(V)$ of which $\mc U^+(V)$ contains forms which are of type $(n,0)$ with respect to some compatible complex structure on $V$.
Some first results about $\mc U(V)$ are collected in the following theorem.

\begin{theorem}
$\mc U^+(V)$ is an open semialgebraic subset of the space of primitive forms $\Lambda^n_\mathrm{pr}V^\vee$, has the homotopy type of $S^1$, and $\mathrm{Sp}(V)$ acts properly on it.
\end{theorem}

See Propositions~\ref{prop_semialg}, \ref{PropUVTopology}, and \ref{prop_SpActsProperly} in the main text.
We also determine the structure of $\mc U^+(\RR^2)$ and $\mc U^+(\RR^4)$ completely, see Subsection~\ref{subsec_lowdim}, and classify elements of $\mc U^+(\RR^6)$ in terms of pairs of compatible complex structures.
This is based on prior work of Hitchin on three-forms~\cite{hitchin_3forms}.
Conjecturally, the universal cover of $\mc U^+(V)$ is (perhaps a component of) the space of stability conditions on the Fukaya category of a torus $V/\Lambda$ with constant rational symplectic form.
When $\dim V=2,4$ this follows from results of Bridgeland~\cite{bridgeland07,bridgeland_k3}, and proofs of homological mirror symmetry in these cases~\cite{pz, abouzaid_smith}.

More generally, one may consider a sort of ``soft'' version of Calabi--Yau geometry given by a symplectic manifold $(M,\omega)$ with a closed complex--valued middle degree form $\Omega$ such that $\Omega_p\in\mc U^+(T_pM)$ for all $p\in M$, i.e. which does not vanish on any Lagrangian subspace of any tangent space. 
Some basic definitions in this direction are collected in Section~\ref{sec_global}.
This type of structure was suggested by Kontsevich as a geometric approach to spaces of stability conditions, in particular as a way of understanding the origin of non--geometric stability conditions.
By contrast, a geometric stability condition is one coming from Calabi--Yau geometry by the conjectural construction as described by Joyce~\cite{joyce_conj}.

In the final section we study $\mc U^+(\RR^{2n})/\mathrm{Sp}(2n,\ZZ)$, which generalizes the moduli space of flat tori $\mc U^+(\RR^{2})/\mathrm{SL}(2,\ZZ)$ and contains the moduli space of principally polarized abelian varieties.
Our first result in this direction concerns the existence of a systolic bound, i.e. an inequality of the form
\begin{equation*}
\left(\mathrm{Sys}_{\Omega}(T^{2n})\right)^2\leq C\mathrm{Vol}_\Omega(T^{2n})
\end{equation*}
where $\mathrm{Sys}_{\Omega}(T^{2n})$ is the minimum of the volumes of compact special Lagrangians in the torus and $\mathrm{Vol}_\Omega(T^{2n})$ is the volume computed by integrating the form $\Omega\wedge\overline{\Omega}$.
These quantities were considered in the context of Calabi--Yau manifolds and stability conditions by Fan--Kanazawa--Yau~\cite{fky}, and Fan \cite{fan_systole}

\begin{theorem}
Let $(V/\Lambda,\omega)$ be a rational symplectic torus, then a systolic bound holds on the ``almost geometric'' subset $\mc U_\mathrm{ag}(V)\subset\mc U^+(V)$, i.e. those forms, $\Omega$, which are complex linear up to the action of $\mathrm{GL}^+(2,\RR)$.
In particular, a systolic bound holds on all of $\mc U^+(\RR^4)$.
\end{theorem}

See Theorem~\ref{thm_torus_sys_ag} in the main text.
Our second result in the final section of the text (Theorem~\ref{thm_vol_inf}) shows that the volume of $\mc U^+_1(\RR^6)/\mathrm{Sp}(6,\ZZ)$ (with respect to its natural $\mathrm{SL}(2,\RR)$--invariant volume form) is infinite, in contrast to the case of $\mc U_1^+(\RR^2)/\mathrm{SL}(2,\ZZ)$.
Here, $\mc U^+_1(V)\subset\mc U^+(V)$ is the subset of forms which are suitably normalized to unit volume.

\begin{theorem}
$\mathrm{Vol}\left(\mc U_1^+(\RR^6)/\mathrm{Sp}(6,\ZZ)\right)=\infty$
\end{theorem}

We believe there there is still much to be said about the spaces $\mc U(V)$ and open questions are indicated throughout.

\subsection*{Acknowledgements}
The author would like to thank Yu-Wei Fan, Maxim Kontsevich, Pranav Pandit, Hiro Lee Tanaka, and Alex Wright for valuable discussions.

\section{Linear theory}
\label{sec_lin_theory}

In this section we define and study the spaces $\mc U(V)$ assigned to a symplectic vector space $V$ over $\RR$.
A basic difficulty is that the dimension of these spaces (which are Catalan numbers) grow exponentially with increasing dimension of $V$.
For example, while it follows from very general principles that each $\mc U(V)$ is described by finitely many inequalities, we know of a useful explicit set of inequalities only if $\dim_\RR V\leq 4$, see Subsection~\ref{subsec_lowdim}.
Nevertheless, we are able to establish some general properties $\mc U(V)$, such as its topology (Subsection~\ref{TopUn}), functoriality properties (Subsection~\ref{SubSecFunctoriality}), and stability of orbits under the $\mathrm{Sp}(V)$--action (Subsection~\ref{subsec_action}).

\subsection{Primitive forms}

We recall some standard symplectic linear algebra.
Let $(V,\omega)$ be a symplectic vector space of dimension $2n$.
The nilpotent Lefschetz operator $L(\alpha):=\omega\wedge\alpha$ extends to a linear representation of the Lie algebra $\mathfrak{sl}(2)$ on exterior forms, $\Lambda^\bullet V^\vee$.
Concretely, this implies that any $k$-form, $\alpha$, on $V$ has a unique decomposition
\begin{equation}
\alpha=\alpha_k+\omega\wedge\alpha_{k-2}+\omega^2\wedge\alpha_{k-4}+\ldots
\end{equation}
where each $\alpha_i$ is a \textit{primitive} $i$-form, i.e. $\omega^{-1}\llcorner \alpha_i=0$ or equivalently $\omega^{n-i+1}\wedge\alpha_i=0$.
For $0\leq k\leq n$ we write $\Lambda^k_{\mathrm{pr}}V^\vee\subset\Lambda^kV^\vee$ for the subspace of primitive $k$-forms.
Its dimension is
\begin{equation}
\dim\Lambda^k_{\mathrm{pr}}V^\vee=\binom{2n}{k}-\binom{2n}{k-2}
\end{equation}
and in particular $\dim\Lambda^n_{\mathrm{pr}}V^\vee=C_{n+1}$ is the $(n+1)$st Catalan number.

A set of vectors $v_1,\ldots,v_k\in V$ spans an isotropic subspace of dimension $k$ if and only if the $k$-vector $v_1\wedge\ldots\wedge v_k\in\Lambda^kV$ is non-zero and primitive.
Thus, the classical Pl\"ucker embedding restricts to an embedding
\begin{equation}
\mathrm{Gr}_{iso}(k,V)\hookrightarrow \mathbb P\left(\Lambda^k_\mathrm{pr}V\right)
\end{equation} 
of the Grassmannian of isotropic $k$-planes in $V$ into projective space. 
The image is the set of decomposable $k$-vectors which is cut out by the quadratic Pl\"ucker relations, giving $\mathrm{Gr}_{iso}(k,V)$ the structure of a real algebraic variety.

\subsection{A non-vanishing condition}

Any $\RR$--valued $n$-form, $\Omega$, on a $2n$-dimensional symplectic vector space, $V$, (over $\RR$) must vanish on some Lagrangian subspace for topological reasons.
On the other hand, if we choose a compatible complex structure on $V$ and let $\Omega$ be a non-zero complex valued form of type $(n,0)$ or $(0,n)$ then 
\begin{equation}
\Omega|_L\neq 0\in \left(\Lambda^n L^\vee\right)\otimes\CC\cong\CC
\end{equation} 
for any Lagrangian subspace $L\subset V$.
This is easy to see using, for example, the fact that the unitary group, $\mathrm{U}(V)$, acts transitively on the set of Lagrangian subspaces and $\Omega$ transforms as a character of $\mathrm{U}(V)$.
More generally we consider the following types of forms.

\begin{df}
\label{def_U}
A form $\Omega\in\left(\Lambda^n_\mathrm{pr}V^\vee\right)\otimes\CC$ is \textbf{non-vanishing on Lagrangian subspaces} if its pullback to any Lagrangian subspace $L\subset V$ is non-zero.
Let
\begin{equation}
\mc U(V)\subset\left(\Lambda^n_\mathrm{pr}V^\vee\right)\otimes\CC
\end{equation}
be the set of such forms on $V$ and let $\mc U(n):=\mc U(\RR^{2n})$.
By definition we have $\mc U(0)=\CC^*$.
We denote by 
\begin{equation}
\mc U_{\mathrm{geom}}(V)\subset \mc U(V) 
\end{equation} 
the subset of forms which are of type $(n,0)$ for some compatible complex structure on $V$.
\end{df}

\begin{remark}
Note that in the definition of $\mc U(V)$ one could have replaced $\CC$ by any two-dimensional vector space over $\RR$.
The definition of $\mc U_{\mathrm{geom}}(V)$, on the other hand, makes use of the field structure on $\CC$.
\end{remark}

This non-vanishing condition is open and semi-algebraic.

\begin{prop}\label{prop_semialg}
$\mc U(V)$ is an open semi-algebraic subset of $\left(\Lambda^n_\mathrm{pr}V^\vee\right)\otimes\CC$.
\end{prop}

\begin{proof}
Let $\mathrm{LGr}(V)=\mathrm{Gr}_{iso}(n,V)$ be the Lagrangian Grassmannian of $V$.
Consider the set
\begin{equation}
Z:=\left\{(\Omega,L)\in\left(\Lambda^n_\mathrm{pr}V^\vee\otimes\CC\right)\times\mathrm{LGr}(V)\mid \Omega|_L=0\right\}
\end{equation}
which is closed and algebraic in $\left(\Lambda^n_\mathrm{pr}V^\vee\otimes\CC\right)\times\mathrm{LGr}(V)$.
Its projection to the first factor, $\mathrm{pr}_1(Z)$, is closed by compactness of $\mathrm{LGr}(V)$ and semi-algebraic by the Tarski--Seidenberg theorem (quantifier elimination for semialgebraic sets).
The statement follows since 
\begin{equation}
\left(\Lambda^n_\mathrm{pr}V^\vee\otimes\CC\right)\setminus \mathrm{pr}_1(Z)=\mc U(V).
\end{equation}
\end{proof}

As a consequence, $\mc U(V)$ has the structure of a complex manifold and inherits the indefinite Hermitian metric
\begin{equation}
\langle \Omega_1,\Omega_2\rangle = \overline{\Omega}_1\wedge\Omega_2\left(\frac{\omega^n}{n!}\right)^{-1}
\end{equation}
from $\Lambda^n_\mathrm{pr}V^\vee\otimes\CC$.
Also, $\mc U(V)$ has a natural right action of $\mathrm{Sp}(V)$ and a left action of $\mathrm{GL}(2,\RR)$ coming from its action on $\CC=\RR^2$, and these actions commute.
However, $\mc U(V)$ has much larger dimension than $\mathrm{Sp}(V)$ in general, so cannot be a homogeneous space for that group.

\subsection{Topology of $\mc U(n)$}
\label{TopUn}

The topology of $\mc U(n)$ turns out to be quite simple. 
We will show in this subsection that for $n\geq 1$ there are two connected components, $\mc U^\pm(n)$, both homotopic to $S^1$.

Any $\Omega\in\mc U(n)$ gives a circle--valued ``phase'' map
\begin{equation}\label{CircleValuedMap}
\phi_\Omega:\mathrm{LGr}(V)\to\RR/\pi\ZZ,\qquad L\mapsto \mathrm{Arg}\left(\Omega|_L\right)
\end{equation}
inducing a map between fundamental groups.
Now, $\pi_1(\mathrm{LGr}(V))\cong\mathbb Z$ canonically.
For $V=\CC^n$ the canonical generator is represented by the loop
\begin{equation}
\RR/\pi\ZZ\to \mathrm{LGr}(\CC^n),\qquad \theta\mapsto e^{i\theta}\RR\times\RR^{n-1}
\end{equation}
which $\phi_\Omega$ sends to
\begin{equation}
\mathrm{Arg}\left(c_1\cos(\theta)+c_2\sin(\theta)\right)\in\RR/\pi\ZZ
\end{equation}
where $c_1,c_2\in\CC$ are coefficients of $\Omega$:
\begin{equation} \label{OmegaOnStdLoop}
\Omega=c_1dx_1\wedge\ldots\wedge dx_n+c_2dy_1\wedge dx_2\wedge\ldots\wedge dx_n+\ldots
\end{equation}
Thus, depending on the sign of $\mathrm{Re}(\overline{c_1}c_2)$, $\phi_\Omega$ sends the canonical generator either to the positive or negative generator of $\pi_1(\RR/\pi\ZZ)$.

\begin{df}
If $\dim V>0$ let $\mc U^+(V)$ (resp. $\mc U^-(V)$) be the set of those $\Omega\in\mc U(V)$ which send the canonical generator of $\pi_1(\mathrm{LGr}(V))$ to a positive (resp. negative) loop in $\pi_1(\RR/\pi\ZZ)$.
\end{df}

Note that if $\Omega$ is of type $(n,0)$ with respect to some complex structure on $V$ then it is in $\mc U^+(V)$, i.e. $\mc U_\mathrm{geom}(V)\subset\mc U^+(V)$.
Complex conjugation interchanges $\mc U^+(V)$ and $\mc U^-(V)$.
The next proposition show that $\mc U^\pm(V)$ are the connected components of $\mc U(V)$.

\begin{prop}\label{PropUVTopology}
Let $V$ be a symplectic vector space with choice of compatible almost complex structure. 
Then $\mc U^+(V)$ deformation retracts to $\left(\Lambda^{n,0}V^\vee\right)\setminus\{0\}$.
In particular, $\mc U^+(V)$ is homotopic to $S^1$.
\end{prop}

\begin{proof}
We will first show that if $\Omega\in\mc U^+(V)$, then the $(n,0)$-component of $\Omega$ is non-zero.
Write
\begin{equation}
\Omega=\Omega^{n,0}+\Omega^{n-1,1}+\ldots+\Omega^{0,n}
\end{equation}
where $\Omega^{k,n-k}$ is of type $(k,n-k)$, $k=0,\ldots,n$.
Let $w\in\Lambda^n_{\mathrm{pr}}V$ such that $w$ represents a Lagrangian subspace of $V$ (i.e. is decomposable).
Then by assumption on $\Omega^{k,n-k}$ we have
\begin{equation}
\Omega^{k,n-k}\left(e^{i\theta}w\right)=e^{(2k-n)i\theta}\Omega^{k,n-k}(w)
\end{equation}
for $\theta\in\RR$.
Let $a_k:=\Omega^{k,n-k}(w)\in\CC$ for $k=0,\ldots,n$ and $p(z):=a_nz^n+\ldots+a_1z+a_0$ then
\begin{equation}\label{OmegaAsPoly}
\Omega\left(e^{i\theta}w\right)=\left(e^{i\theta}\right)^{-n}p\left(e^{2i\theta}\right).
\end{equation}
The map 
\begin{equation}
[0,2\pi]\to\mathrm{LGr}(V),\qquad \theta\mapsto e^{i\theta}w
\end{equation}
defines a loop of Maslov index $2n$ in the Lagrangian Grassmannian, so since $\Omega$ is positively oriented the map
\begin{equation}
\RR/2\pi\ZZ\to \mathrm{U}(1),\qquad \theta\mapsto \Omega(e^{i\theta}w)
\end{equation}
has degree $n$. 
Looking at \eqref{OmegaAsPoly} this implies that $p$ has degree $n$ with all roots in the open unit disk.
In particular, $\Omega^{n,0}\neq 0$.

For $\Omega\in\mc U^+(V)$ consider
\begin{equation}
\Omega_t:=\Omega^{n,0}+t\Omega^{n-1,1}+\ldots+t^n\Omega^{0,n},\qquad t\in[0,1]
\end{equation}
then 
\begin{equation}
\Omega_1=\Omega,\qquad \Omega_0\in \left(\Lambda^{n,0}V^\vee\right)\setminus\{0\}
\end{equation}
and we claim that $\Omega_t\in\mc U^+(V)$ for all $t\in [0,1]$.
This follows since if $w$, $p$ are as above then
\begin{equation}
\Omega_t(w)=t^np(t^{-1})\neq 0
\end{equation}
as $p$ has no zeros outside the unit disk.
\end{proof}

We can also say more about the maps $\phi_\Omega$.

\begin{coro}
If $\Omega\in\mc U(V)$, then the map $\phi_\Omega:\mathrm{LGr}(V)\to\RR/\pi\ZZ$ defined in \eqref{CircleValuedMap} is a fiber bundle with fibers diffeomorphic to the special Lagrangian Grassmannian $\mathrm{SU}(n)/\mathrm{SO}(n)$, where $2n=\dim V$.
\end{coro}

\begin{proof}
The explicit formula \eqref{OmegaOnStdLoop} shows that $\phi_\Omega$ has no critical points, thus gives a fiber bundle.
For $\Omega\in \mc U(V)$ of type $(n,0)$ or $(0,n)$ it is clear that the fibers are diffeomorphic to $\mathrm{SU}(n)/\mathrm{SO}(n)$. 
Since $\mc U^\pm(V)$ are path connected, the claim follows.
\end{proof}

\subsection{Low dimensions}
\label{subsec_lowdim}

\subsubsection{$\mc U(1)$}

If $\dim_\RR V=2$, then $\mc U(V)$ is just the set of $\RR$-linear isomorphisms $V\to\CC$ and $\Omega\in \mc U^+(V)$ iff it is orientation preserving, where the orientation on $V$ is determined by the symplectic structure.
From this it is also clear that any $\Omega\in \mc U^+(V)$ is complex linear for a unique choice of compatible complex structure on $V$, i.e. all forms in $\mc U^+(V)$ are geometric, which happens only for $n=1$.
The group $\mathrm{GL}(V)$ acts freely and transitively on $\mc U(V)$ and $\mathrm{GL}^+(V)$ acts freely and transitively on $\mc U^+(V)$.

\subsubsection{$\mc U(2)$}
\label{SubSubSecU2}

We suppose $\dim_{\RR}V=4$, so $\dim_{\CC}\mc U(V)=5$.
The 5--dimensional space $\Lambda^2_{\mathrm{pr}}V$ of primitive bivectors has a non-degenerate symmetric bilinear form of signature $++---$ given by
\begin{equation}
\langle \alpha,\beta \rangle\frac{\omega^2}{2}:=\alpha\wedge\beta.
\end{equation}
We will use it to identify $\Lambda^2_{\mathrm{pr}}V$ with its dual, $\Lambda^2_{\mathrm{pr}}V^\vee$.
For $V=\CC^2$ an orthonormal basis of this space is given by
\begin{gather}
\nonumber \frac{1}{\sqrt{2}}\mathrm{Re}\left(dz_1\wedge dz_2\right),\quad \frac{1}{\sqrt{2}}\mathrm{Im}\left(dz_1\wedge dz_2\right) \\ \nonumber
\frac{1}{\sqrt{2}}\mathrm{Re}\left(dz_1\wedge d\bar{z}_2\right),\quad \frac{1}{\sqrt{2}}\mathrm{Im}\left(dz_1\wedge d\bar{z}_2\right),\quad 
\frac{i}{2\sqrt{2}}\left(dz_1\wedge d\bar{z}_1-dz_2\wedge d\bar{z}_2\right).
\end{gather}
A bivector $0\neq v\in \Lambda^2_{\mathrm{pr}}V$ comes from a Lagrangian plane in $V$ if and only if $v\wedge v=0$ (Pl\"ucker relation).
Thus, the affine cone over the Lagrangian Grassmannian $\mathrm{LGr}(V)\subset\PP(\Lambda^2_{\mathrm{pr}}V)$ is just the light-cone in $\Lambda^2_{\mathrm{pr}}V$.
This allows us to give explicit inequalities for $\mc U(2)$.

\begin{prop}
A form $\Omega=\alpha+\beta i\in \left(\Lambda^2_{\mathrm{pr}}V^\vee\right)\otimes \CC$ is non-vanishing on Lagrangian subspaces if and only if $\alpha$ and $\beta$ span a spacelike plane in $\Lambda^2_{\mathrm{pr}}V^\vee$, i.e. the symmetric matrix
\begin{equation}
S_\Omega:=\begin{pmatrix} \langle\alpha,\alpha\rangle & \langle\alpha,\beta\rangle \\ \langle\beta,\alpha\rangle & \langle\beta,\beta\rangle \end{pmatrix}
\end{equation}
is positive definite.
\end{prop}

\begin{proof}
Identify $\Lambda^2_{\mathrm{pr}}V\cong\Lambda^2_{\mathrm{pr}}V^\vee$ as before and consider $\alpha,\beta$ as vectors in the former.
A Lagrangian subspace, $L$, is represented by a lightlike bivector $v\neq 0$ and $\Omega$ vanishes on $L$ iff $v$ is orthogonal to both $\alpha$ and $\beta$.
Thus, $\Omega$ is does not vanish on any Lagrangian subspace iff the orthogonal complement $P^\perp$ to the subspace $P\subset \Lambda^2_{\mathrm{pr}}V$ spanned by $\alpha$ and $\beta$ intersects the lightcone in the origin only.
This happens iff the restriction of the symmetric form to $P^\perp$ is definite, which implies $\dim_{\RR}P^\perp=3$ and the symmetric form is negative definite on $P^\perp$. 
Thus $\dim_{\RR}P=2$ and the symmetric form is positive definite on $P$.
\end{proof}

It turns out the the matrix $S_\Omega$ which appeared in the previous proposition is a complete invariant of $\Omega\in\mc U^+(V)$ under the $\mathrm{Sp}(V)$--action.

\begin{prop}
The map
\begin{equation}
\mc U^+(V)\to \mathrm{Mat}(2\times 2,\RR), \quad \Omega\mapsto S_\Omega
\end{equation}
induces a homeomorphism from the quotient space $\mc U^+(V)/\mathrm{Sp}(V)$ to the space $PDS(2,\RR)$ of positive definite symmetric $2\times 2$ matrices over $\RR$.
\end{prop}

\begin{proof}
The symplectic group $\mathrm{Sp}(V)$ fixes $\omega$ by definition and thus preserves the symmetric form on $\Lambda^2_{\mathrm{pr}}V^\vee$.
In fact the image of $\mathrm{Sp}(V)$ in the orthogonal group $\mathrm{O}\left(\Lambda^2_{\mathrm{pr}}V^\vee\right)\cong \mathrm{O}(2,3)$ is the connected component of the identity, $\mathrm{SO}^+(2,3)$, a consequence of the coincidence of root systems $B_2=C_2$.
Since $\mathrm{SO}^+(2,3)$ acts transitively on spacelike planes in $\Lambda^2_{\mathrm{pr}}V^\vee$, the only invariants of $\Omega$ are the lengths of $\alpha$ and $\beta$, and the angle between them.
This is precisely the data recorded by $S_\Omega$.
\end{proof}

It is an easy consequence of the previous proposition that every $\Omega\in\mc U^+(2)$ can be written as
\begin{equation}\label{U2CanonicalForm}
rdz_1\wedge dz_2+cd\bar{z}_1\wedge d\bar{z}_2,\qquad r\in\RR, c\in\CC, r>|c|
\end{equation}
with respect to a suitable symplectic basis.
Note that this form is $\mathrm{SU}(2)$-invariant.
Moreover the double coset space 
\begin{equation}
\mathrm{GL}^+(2,\RR)\backslash\mc U^+(V)/\mathrm{Sp}(V)
\end{equation}
reduces to a point.
This fails for $\dim_\RR V>4$ by a dimension argument.

\subsection{Functoriality}
\label{SubSecFunctoriality}

We again start the subsection with a review of a bit of symplectic linear algebra.
The graph of a linear symplectomorphism $V_1\to V_2$ is a Lagrangian subspace in $V_1\times V_2$ with symplectic form $-p_1^*\omega_1+p_2^*\omega_2$.
More generally one considers arbitrary Lagrangian subspaces of $V_1\times V_2$ with this symplectic form, the \textit{linear Lagrangian correspondences}, as morphisms in a linear symplectic category.
Composition of such morphisms is not continuous on the Lagrangian Grassmannian, and from the modern point of view one should instead use derived intersection, though this will not be necessary for the present discussion.
Any linear Lagrangian correspondence factors into a symplectic quotient, a symplectomorphism, and the inverse of a symplectic quotient \cite{benenti81}.
A linear symplectic quotient (aka reduction) is a Lagrangian correspondence of the form
\begin{equation}
\left\{(w,w+W^\perp)\mid w\in W\right\}\subset V\times (W/W^\perp)
\end{equation}
for some coisotropic subspace $W\subset V$. 
Note that 
\begin{equation}
W^\perp:=\left\{v\in V\mid w\in W\implies\omega(v,w)=0\right\}
\end{equation}
is the symplectic complement and $W$ is coisotropic iff $W^\perp\subset W$.

Fix a coisotropic subspace $W\subset V$ and a non-zero volume element $\nu\in \Lambda^{\mathrm{top}}W^\perp$, then
we get a map
\begin{equation}
\mc U(V)\to \mc U\left(W/W^\perp\right),\qquad \Omega\mapsto\nu\llcorner\Omega\mid_W
\end{equation}
induced by the symplectic quotient correspondence $V\to W/W^\perp$.

Let $(V_1,\omega_1)$ and $(V_2,\omega_2)$ be symplectic vector spaces of dimensions $2n_1$, $2n_2$ respectively, and $\Omega_k\in\left(\Lambda^{n_k}_{\mathrm{pr}}V_k^\vee\right)\otimes\CC$ for $k=1,2$. 
Then $V=V_1\oplus V_2$ is a symplectic vector space with form $\omega:=p_1^*\omega_1+p_2^*\omega_2$, where $p_k:V\to V_k$ are canonical projections, and we can consider
\begin{equation}
\Omega:=p_1^*\Omega_1\wedge p_2^*\Omega_2 \in \left(\Lambda^{n_1+n_2}_{\mathrm{pr}}V^\vee\right)\otimes\CC.
\end{equation}
A natural question at this point is: If $\Omega_k\in \mc U(V_k)$ for $k=1,2$ does it follow that $\Omega\in\mc U(V)$?
This seems to be a somewhat delicate point and we prove only a partial result.
To state it we make the following definition. 

\begin{df}
A form $\Omega\in\mc U^+(V)$ is \textbf{almost geometric} if it is in the $\mathrm{GL}^+(2,\RR)$--orbit of $\mc U_\mathrm{geom}(V)$.
Equivalently, there exists a compatible complex structure on $V$ such that $\Omega$ is a sum of forms of types $(n,0)$ and $(0,n)$, $2n=\dim_\RR V$.
Denote by $\mc U_\mathrm{ag}\subset\mc U^+(V)$ the subset of almost geometric forms.
\end{df}

Note that we have inclusions
\begin{equation}
\mc U_\mathrm{geom}(V)\subset\mc U_\mathrm{ag}(V)\subset \mc U^+(V)\subset \mc U(V)
\end{equation}
which are all strict for sufficiently large dimension of $V$.
In some respects, the almost geometric forms are a more natural class than the geometric ones, c.f. the remark after Definition~\ref{def_U}.
The following theorem also supports this.

\begin{theorem}
Let $V_1$ and $V_2$ be symplectic vector spaces, $\Omega_k\in\mc U^+(V_k)$ for $k=1,2$.
If $\Omega_1$ (or $\Omega_2$) is almost geometric then 
\begin{equation}
p_1^*\Omega_1\wedge p_2^*\Omega_2\in \mc U^+(V_1\oplus V_2)
\end{equation}
where $V_1\oplus V_2$ has symplectic form $p_1^*\omega_1+p_2^*\omega_2$ as above.
\end{theorem}

The analogous statement fails if $\Omega_1\in\mc U^+(V)$ but $\Omega_2\in \mc U^-(V)$, or if $\Omega_1$ and $\Omega_2$ are not required to be primitive forms.

\begin{proof}
Let $L\subset V_1\oplus V_2$ be a Lagrangian subspace, i.e. a Lagrangian correspondence from $V_1$ to $V_2$ with the negative symplectic form.
By the classification of linear Lagrangian correspondences, the projections $C_k:=p_k(L)\subset V_k$ are coisotropic and there is a diagram
\begin{equation}
\begin{tikzcd}
V_1 & C_1  \arrow[hook]{l}\arrow[two heads]{d} & C_2 \arrow[hook]{r}\arrow[two heads]{d} & V_2 \\
    & C_1/C_1^\perp \arrow{r}{\phi}            & C_2/C_2^\perp
\end{tikzcd}
\end{equation}
where $\phi$ is a linear anti-symplectomorphism so that
\begin{equation}
L=\left\{(v_1,v_2)\in C_1\times C_2\mid \phi(v_1+C_1^\perp)=v_2+C_2^\perp\right\}.
\end{equation}
Choose non-zero elements $\nu_k\in\Det(C_k^\perp)$, where $\Det$ denotes the top exterior power.
In view of the exact sequence
\begin{equation}
0 \longrightarrow L \longrightarrow C_1\oplus C_2 \longrightarrow C_2/C_2^\perp \longrightarrow 0
\end{equation}
and 
\begin{equation}
\Det(C_k/C_k^\perp)\otimes \Det(C_k^\perp)\cong \Det(C_k)
\end{equation}
the $\nu_k$ give an identification
\begin{equation}
\Det(L) \cong \Det(C_1)\otimes\Det(C_2)\otimes\left(\Det\left(C_2/C_2^\perp\right)\right)^{-1} \cong \Det\left(C_1/C_1^\perp\right)
\end{equation}
under which $\Omega|_L\in\Det(L)$ corresponds to
\begin{equation}
\pm\left(\nu_1\llcorner\Omega_1|_{C_1}\right)\wedge \phi^*\left(\nu_2\llcorner\Omega_2|_{C_2}\right)\in \Det\left(C_1/C_1^\perp\right).
\end{equation}

We claim that if $\Omega_1$ is almost geometric, then so is any reduction of it, in particular $\nu_1\llcorner\Omega_1|_{C_1}$.
By induction it suffices to show this if $C_1$ is a hyperplane and $\nu_1$ a vector.
Choosing suitable coordinates we can assume 
\begin{equation}
\nu_1=\partial/\partial x_1,\qquad \Omega_1 = c_1dz_1\wedge\ldots\wedge dz_n+c_2d\bar{z}_1\wedge\ldots\wedge d\bar{z}_n
\end{equation}
but then
\begin{equation}
\nu_1\llcorner\Omega_1|_{C_1}=c_1dz_2\wedge\ldots\wedge dz_n+c_2d\bar{z}_2\wedge\ldots\wedge d\bar{z}_n
\end{equation}
which is almost geometric on $\CC^{n-1}$.

By the preceding arguments we have reduced the problem to showing the following:
Let $V$ be a symplectic vector space $\Omega_1\in\mc U_\mathrm{ag}(V)$, $\Omega_2\in\mc U^+(V)$, and $\phi:V\to V$ an anti-symplectomorphism, then $\Omega_1\wedge\phi^*\Omega_2\neq 0$.
Also is suffices to show this for just one particular $\phi$, which we chose to be complex conjugation $v\mapsto \bar{v}$ on $V=\CC^n$.
In suitable coordinates we have
\begin{equation}
\Omega_1=c_1dZ+c_2d\bar{Z},\qquad \Omega_2=c_3dZ+c_4d\bar{Z}+\ldots
\end{equation}
with $|c_1|>|c_2|$, $|c_3|>|c_4|$, and writing $dZ:=dz_1\wedge\ldots\wedge dz_n$, so
\begin{equation}
\Omega_1\wedge\phi^*\Omega_2=\left(c_1c_3\pm c_2c_4\right)dZ\wedge d\bar{Z}\neq 0
\end{equation}
which completes the proof.
\end{proof}

\begin{q}
Given $\Omega_1\in\mc U^+(V_1)\setminus \mc U_\mathrm{ag}(V_1)$ (so in particular $\dim_\RR V\geq 6$), is there a symplectic vector space $V_2$ and $\Omega_2\in\mc U^+(V_2)$ such that
\begin{equation}
p_1^*\Omega_1\wedge p_2^*\Omega_2\notin \mc U(V_1\oplus V_2)?
\end{equation}
\end{q}

\subsection{The action of $\mathrm{Sp}(V)$ on $\mc U(V)$}
\label{subsec_action}

The action of $\mathrm{Sp}(V)$ on $\mc U(V)$ has much better properties than the action on the entire vector space $\left(\Lambda^n_{\mathrm{pr}}V^\vee\right)\otimes\CC$, for example it is proper with stable orbits in the sense of geometric invariant theory.
We also show in this subsection that $\mc U^+(V)$ fibers $\mathrm{Sp}(V)$-equivariantly over the geometric part $\mc U_\mathrm{geom}(V)$.

\begin{prop} \label{prop_SpActsProperly}
The action of $\mathrm{Sp}(V)$ on $\mc U(V)$ is proper.
\end{prop}

\begin{proof}
By definition of a proper group action, we need to show that the map
\begin{equation}\label{MapWeNeedToShowIsProper}
\mathrm{Sp}(V)\times\mc U(V)\to \mc U(V)\times \mc U(V), \qquad (g,\Omega)\mapsto (g^*\Omega,\Omega)
\end{equation}
is proper.

Fix a compatible complex structure on $V$.
In particular this gives a Hermitian metric on $V$ and a Euclidean metric on each Lagrangian subspace $L\subset V$.
Any $\Omega\in \mc U(V)$ then defines a function
\begin{equation}
\log|\Omega|:\mathrm{LGr}(V)\to\RR, \qquad L\mapsto\log\left|\Omega|_L\right|.
\end{equation}
Given $R>0$ consider the compact subset
\begin{equation}
B_R:=\left\{\Omega\in\mc U(V)\mid \left|\log|\Omega|\right|\leq R\right\}
\end{equation}
then 
\begin{equation}
\mc U(V)=\bigcup_{R>0}B_R
\end{equation}
and any compact subset of $\mc U(V)$ is contained in some $B_R$.

The symplectic group has maximal compact subgoup $\mathrm{U}(V)\subset\mathrm{Sp}(V)$.
Let $\mathfrak{sp}(V)=\mathfrak{u}(V)\oplus\mathfrak{p}$ be the Cartan decomposition of the Lie algebra, where $\mathfrak{p}$ consists of symmetric matrices in $\mathfrak{sp}(V)$.
Correspondingly one has the polar decomposition of linear symplectomorphisms, $\mathrm{Sp}(V)=\mathrm U(V)\cdot\exp(\mathfrak{p})$.
Moreover each $X\in\mathfrak{p}$ can be diagonalized by a unitary change of basis.

Let $X\in\mathfrak{p}$, then we may assume that $X$ is diagonal of the form
\begin{equation}
X=\mathrm{diag}(\lambda_1,\ldots,\lambda_n,-\lambda_1,\ldots,-\lambda_n),\qquad \lambda_1\geq\ldots\geq\lambda_n\geq 0
\end{equation}
with respect to some orthonormal symplectic basis.
Let $L=\RR^n$, embedded as coordinate subspace for the first $n$ coordinates, then
\begin{equation}
\log|\Omega\cdot\exp(X)|_L|=\log|\Omega|_L|+\lambda_1+\ldots+\lambda_n.
\end{equation}
It follows that if $\Omega\in B_R$ and $\Omega\cdot \exp(X)\in B_R$ then 
\begin{equation}
\lambda_1+\ldots+\lambda_n\leq 2R
\end{equation}
which establishes a bound on $X\in\mathfrak{p}$.
This shows that the preimage of $B_R\times B_R$ under the map \eqref{MapWeNeedToShowIsProper} is compact.
Since any compact subset of $\mc U(V)\times \mc U(V)$ is contained in $B_R\times B_R$ for sufficiently large $R$, this completes the proof.
\end{proof}

\begin{coro}
The stabilizer of any $\Omega\in\mc U(V)$ is a compact subgroup in $\mathrm{Sp}(V)$.
\end{coro}

Next, we study the action of $\mathrm{Sp}(V)$ from the point of view of (real) geometric invariant theory as developped in e.g. \cite{richardson_slodowy}, \cite{boehm_lafuente}.

\begin{prop}
The domain $\mc U(V)$ is a union of closed orbits of the action of $\mathrm{Sp}(V)$ on $\left(\Lambda^n_\mathrm{pr}V^\vee\right)\otimes\CC$. 
Fix a Euclidean inner product on $\left(\Lambda^n_\mathrm{pr}V^\vee\right)\otimes\CC$ such that $\mathrm U(V)$ acts by isometries.
Let $\Omega_0\in \mc U(V)$ and $O:=\Omega_0\cdot\mathrm{Sp}(V)$ the orbit containing $\Omega_0$. 
Then each critical point of the function 
\begin{equation}
S:O\to\RR,\qquad \Omega\mapsto \|\Omega\|^2
\end{equation}
is a minimum and the set of minima is a single $\mathrm U(V)$--orbit.
\end{prop}

\begin{proof}
Let $\mathfrak{sp}(V)=\mathfrak{u}(V)\oplus\mathfrak{p}$ be the Cartan decomposition as before.
The proof is essentially an application of geometric invariant theory for real reductive groups.
To prove closedness of orbits in $\mc U(V)$ we will check a version of the Hilbert--Mumford criterion.
More precisely, let $\Omega_0\in\mc U(V)$ and $X\in\mathfrak{p}$, then we need to show that if the limit
\begin{equation}
\lim_{t\to+\infty}\Omega_0\exp(tX)
\end{equation} 
exists, then it is contained in the same orbit, $\Omega_0\cdot\mathrm{Sp}(V)$.

To see this let $X\in\mathfrak{p}$, then we may assume that $X$ is diagonal of the form
\begin{equation}
X=\mathrm{diag}(\lambda_1,\ldots,\lambda_n,-\lambda_1,\ldots,-\lambda_n),\qquad \lambda_1\geq\ldots\geq\lambda_n\geq 0
\end{equation}
with respect to some orthonormal symplectic basis.
The coefficient of $dx_1\wedge\ldots\wedge dx_n$ in $\Omega_0\exp(tX)$ is
\begin{equation}
ce^{(\lambda_1+\ldots+\lambda_n)t}
\end{equation}
for some $c\in\CC$. 
In fact $c\neq 0$ by the assumption that $\Omega_0$ is non-vanishing on Lagrangian subspaces.
It follows that the limit exists only if $\lambda_k=0$ for all $k$, i.e. $X=0$.
\end{proof}

\begin{coro}
For given $\Omega\in\mc U(V)$ there is a unique compatible complex structure, $J$, on $V$ which minimizes $\|\Omega\|^2$ for the inner product on $\left(\Lambda^n_\mathrm{pr}V^\vee\right)\otimes\CC$ induced by $J$.
The stabilizer of $\Omega$ in $\mathrm{Sp}(V)$ is a compact subgroup of $\mathrm U(V)$ where $V$ is equipped with the minimizing complex structure.
\end{coro}

Let $\mathcal J$ be the space of compatible complex structures on $V$. 
For a given choice of $J_0\in\mathcal J$ we obtain an identification $\mathcal J=\mathrm{Sp}(V)/\mathrm U(V)$.
Let $\mathcal G$ be the $\CC^*$-bundle over $\mathcal J$ with fiber the complex-linear volume forms, $\Lambda^{n,0}V^\vee\setminus \{0\}$. 
By the above corollary, the natural map $\mathcal G\to\mc U(V)$, whose image are the \textit{geometric} elements of $\mc U(V)$, is an injection.
Furthermore we have $\mathrm{Sp}(V)$ equivariant maps
\begin{equation}
\mc U^+(V)\longrightarrow \mc U_\mathrm{geom}(V)=\mathcal G\longrightarrow\mathcal J
\end{equation}
where the first map sends $\Omega\in\mc U^+(V)$ to its $(n,0)$-component with respect to the minimizing complex structure.

\subsection{$\mc U(3)$}

Suppose $\dim_{\RR}V=6$, thus $\dim_{\CC}\mc U(V)=14$ and $\dim_\CC\mc U_\mathrm{geom} (V)=7$.

The classification of alternating trilinear forms on a six-dimensional vector space goes back to Reichel~\cite{reichel1907}.
Restricting to primitive forms and the action of $\mathrm{Sp}(6)\subset \mathrm{GL}(6)$, the problem was solved for algebraically closed fields of characteristic zero by Igusa~\cite{igusa70}.
Over $\RR$, the classification is contained in work of Lychagin--Rubtsov--Chekalov~\cite{lychagin93}.
For given $\alpha\in \Lambda^3_\mathrm{pr}V^\vee$ the key is to consider the symmetric bilinear form, $q_\alpha$, on $V$ given by
\begin{equation}\label{ThreeFormQ}
q_\alpha(X,Y)\frac{\omega^3}{3!}=(X\llcorner\alpha)\wedge(Y\llcorner\alpha)\wedge\omega.
\end{equation}
In fact the map 
\begin{equation}
\alpha\mapsto q_\alpha,\qquad \Lambda^3_\mathrm{pr}V^\vee\to\mathrm{Sym}^2(V^\vee)\cong \mathfrak{sp}(V)
\end{equation}
has an interpretation as a moment map for the $\mathrm{Sp}(V)$ action on $\Lambda^3_\mathrm{pr}V^\vee$.
Since $V$ has a natural volume element, $\omega^3/3!$, the determinant of $q_\alpha$ is a well defined real number, in fact the only continuous invariant of $\alpha$.
Furthermore, the Pl\"ucker relations for $\mathrm{LGr}(V)$ are just $q_\alpha=0$.

We are interested in the case when $\alpha$ is the real or imaginary part of $\Omega\in\mc U(V)$, in which case $q_\alpha$ turns out to be positive definite. 

\begin{lemma}\label{lem_Qposdef}
Let $\dim_\RR V=6$ and $\Omega=\alpha+i\beta\in\mc U(V)$, then $q_\alpha$ defined in~\eqref{ThreeFormQ} is positive definite.
\end{lemma}

\begin{proof}
Let $v\in V$, $v\neq 0$.
Consider the kernel, $N\subset V$, of the alternating 2-form $v\llcorner \alpha$. 
Any alternating 2-form has even rank, so $N$ must be even-dimensional.
Furthermore $v\in N$, and we claim that $N\subset V$ is a symplectic subspace with $\dim_\RR N=2$.
Otherwise, i.e. if $N$ is Lagrangian or $\dim_\RR N>2$, we could find a vector $w\in N$, linearly independent from $v$, such that $\omega(v,w)=0$.
But then there is a non-contractible loop, $L_t$, in $\mathrm{LGr}(V)$ of Lagrangian subspaces which all contain $v$ and $w$ and on which $\alpha$ vanishes.
This forces the loop $\mathrm{Arg}\left(\Omega|_{L_t}\right)$ to be constant, contradicting the assumption $\Omega\in\mc U(V)$, c.f. the discussion at the beginning of Subsection~\ref{TopUn}.

Let $W\subset V$ be the symplectic complement of the line spanned by $v$, so $W$ is coisotropic of dimension five.
As in Subsection~\ref{SubSecFunctoriality} we get 
\begin{equation}
\alpha'+i\beta':=v\llcorner\Omega|_W\in\mc U(W/W^\perp)
\end{equation}
and from Subsection~\ref{SubSubSecU2} we know that $\alpha',\beta'$ span a positive definite 2-plane in $\Lambda^2_\mathrm{pr}(W/W^\perp)^\vee$, in particular $\alpha'\wedge\alpha'$ is a positive multiple of the square of the symplectic form on $W/W^\perp$.
But $N^\perp\subset W$ and $W/W^\perp$ may be identified with $N^\perp$ as a symplectic vector space.
Under this identification we have $v\llcorner\alpha=\alpha'$ and so $(v\llcorner\alpha)\wedge(v\llcorner\alpha)\wedge \omega$ must be a positive multiple of $\omega^3$.
\end{proof}

The classification result of \cite{lychagin93} implies that if $\alpha,\beta\in\Lambda^3_\mathrm{pr}V^\vee$ with $q_\alpha,q_\beta$ positive definite, then $\alpha$ and $\beta$ lie in the same $\mathrm{Sp}(V)$ orbit if and only if $\det(q_\alpha)=\det(q_\beta)$.
In particular, if $q_\alpha$ is positive definite then there is a compatible almost complex structure on $V$ and $\Omega\in\Lambda^{3,0}V^\vee$ such that $\alpha=\mathrm{Re}(\Omega)$.
We have proven the following.

\begin{prop}\label{PropU3PairGeometric}
Let $\dim_\RR V=6$ and $\Omega\in\mc U(V)$, then there is a pair $J_1,J_2$ of compatible complex structures on $V$ and $\Omega_1,\Omega_2\in\left(\Lambda^3_\mathrm{pr}V^\vee\right)\otimes\CC$ such that $\Omega_k$ is of type $(3,0)$ with respect to $J_k$ and
\begin{equation}
\Omega=\mathrm{Re}(\Omega_1)+i\mathrm{Im}(\Omega_2).
\end{equation}
\end{prop}

If $J$ is any complex structure on $\CC^n$ which is compatible with its standard symplectic structure, then there is a unitary change of coordinates (preserving the standard hermitian structure) such that
\begin{equation}
J=\left(
\begin{array}{c|c}
0 & \begin{array}{ccc}
    -\lambda_1^{-1} & & \\
    & \ddots & \\
    & & -\lambda_n^{-1}
  \end{array} \\
\hline
\begin{array}{ccc}
    \lambda_1 & & \\
    & \ddots & \\
    & & \lambda_n
  \end{array} & 0
\end{array}
\right)
\end{equation}
with $0<\lambda_1\leq\ldots\leq\lambda_n\leq 1$.
In these coordinates the 1-forms $\lambda_kdx_k+idy_k$ are of type $(1,0)$ with respect to $J$.

\begin{coro}
Suppose $\Omega\in\mc U(\CC^3)$, then after a suitable linear symplectic change of coordinates we get
\begin{align*}
\Omega=&\mathrm{Re}\left(c_1dz_1\wedge dz_2\wedge dz_3\right) \\
&+i\mathrm{Im}\left(c_2(\lambda_1dx_1+idy_1)\wedge(\lambda_2dx_2+idy_2)\wedge(\lambda_3dx_3+idy_3)\right)
\end{align*} 
for some $c_1,c_2\in\CC^*$ and $0<\lambda_1\leq\lambda_2\leq\lambda_3\leq 1$.
\end{coro}

\begin{prop}
Let $\Omega\in\mc U(3)$, then $\Omega\wedge\overline{\Omega}\neq 0$ and furthermore $i\Omega\wedge\overline{\Omega}$ is a positive (resp. negative) multiple of $\omega^3$ if $\Omega\in\mc U^+(3)$ (resp. $\Omega\in\mc U^-(3)$).
\end{prop}

\begin{proof}
According to Proposition~\ref{PropU3PairGeometric} we can write $\Omega=\alpha+\beta$ with
\begin{gather}
\alpha=c_1dZ+\bar{c}_1d\bar{Z} \\
\beta=c_2dZ-\bar{c}_2d\bar{Z}+\left(\text{components of type }(2,1)\text{ and }(1,2)\right)
\end{gather}
where $c_1,c_2\in\CC$ and $dZ:=dz_1\wedge dz_2\wedge dz_3$.
Suppose $\Omega\in\mc U^+(3)$, then the proof of Proposition~\ref{PropUVTopology} shows that $c_1+c_2\neq 0$.
In fact slightly more is shown, namely that $c_1+c_2$ and $\bar{c}_1-\bar{c}_2$ are the leading and constant coefficients of a polynomial with all roots in the open unit disk.
Thus
\begin{equation}
|c_1+c_2|>|c_1-c_2|
\end{equation}
which is equivalent to $\mathrm{Re}(c_1\bar{c}_2)>0$.
On the other hand,
\begin{equation}
i\Omega\wedge\bar{\Omega}=-2i\alpha\wedge\beta=4i\mathrm{Re}(c_1\bar{c}_2)dZ\wedge d\bar{Z}
\end{equation}
and 
\begin{equation}
\frac{\omega^3}{3!}=\frac{i}{8}dZ\wedge d\bar{Z}.
\end{equation}
This completes the proof since complex conjugation interchanges $\mc U^\pm(3)$ and the sign of $\Omega\wedge\bar{\Omega}$.
\end{proof}


\section{Global theory}
\label{sec_global}

In this section we define a global version of the structure considered in the previous section.
More precisely, we consider symplectic manifolds equipped with a closed and primitive complex--valued middle--degree form, $\Omega$, which satisfies the non--vanishing condition studied in the previous section pointwise.
When $M$ is compact we associate with such an $\Omega$ a volume and systole (Subsection~\ref{subsec_volsys}).
These definitions will be applied to the special case when $M$ is a rational symplectic torus in the next section.
Finally, we discuss a relaxed form of the closedness condition on $\Omega$.

\subsection{$\CC$--polarizations}

\begin{df}\label{def_Cpol}
Let $(M,\omega)$ be a compact symplectic manifold of dimension $2n$.
A \textbf{$\CC$--polarization} on $X$ is a complex-valued middle degree form 
\begin{equation}
\Omega\in \Gamma\left(M,\left(\Lambda^n T^*M\right)\otimes \CC\right)
\end{equation}
such that 
\begin{enumerate}
\item
$\Omega$ is primitive: $\Omega\wedge\omega=0$,
\item
$\Omega$ is closed: $d\Omega=0$,
\item
$\Omega_p\in \Lambda^nT^*_pM$ is non-vanishing on Lagrangian subspaces and 
positively oriented for each $p\in M$, i.e. $\Omega_p\in\mc U^+(T_pM)$.
\end{enumerate}
\end{df}

Although this definition considerably relaxes the Calabi--Yau condition, it is perhaps still too strict and we will discuss potential generalizations later.
Here are some examples of $\CC$--polarizations.
\begin{enumerate}
\item
Let $V$ be a symplectic vector space with lattice $\Lambda\in V$ and $\Omega\in\mc U^+(V)$, then we can consider $\Omega$ as a constant differential form on the symplectic torus $V/\Lambda$.
\item
If $(M,\omega,J)$ is a K\"ahler manifold with non-vanishing holomorphic volume form $\Omega$ (not necessarily satisfying the Calabi--Yau condition), then $\Omega$ defines a $\CC$-polarization.
One could attempt to generalize this example by requiring $J$ to be only a compatible almost complex structure, but in fact the condition $d\Omega=0$ for a non-vanishing $(n,0)$ form implies integrability of $J$ by the Newlander--Nirenberg theorem.
\end{enumerate}

Associated with a $\CC$--polarization $\Omega$ on $(M,\omega)$ is a \textbf{central charge}
\begin{equation}
Z:H_n(X;\ZZ)\to\CC,\qquad Z(L):=\int_L\Omega
\end{equation}
and a notion of \textbf{special Lagrangian submanifold} of phase $\phi\in\RR/\pi\ZZ$, i.e. a Lagrangian submanifold $L\subset M$ such that
\begin{equation}
\mathrm{Arg}\left(\Omega\mid_L\right)=\phi.
\end{equation}
Thus, one has all the data required for a stability condition on the Fukaya category, $\mc F(M)$, of $M$ if we define semistable objects to be those which can be represented by an immersed special Lagrangian submanifold.
This is not yet a precise definition, since one probably needs to allow singular Lagrangian submanifolds, and it is unknown which singularities to allow and how to include such objects in $\mc F(M)$.
Checking that this data satisfies the axioms of a stability condition is another matter.
At least one of the axioms, the support property, is easy.
Not coincidentally, the argument below served as motivation for Kontsevich--Soibelman to introduce the support property in~\cite{ks}, see also \cite{ks14}.
For the special case of surfaces it can be found in \cite{hkk}.

\begin{prop}[Support property]
\label{prop_support}
Let $(X,\omega)$ be a compact symplectic manifold, $\dim_\RR X=2n$, with $\CC$--polarization $\Omega$.
Then there is a norm on $H_n(X;\RR)$ and a constant $C>0$ such that
\begin{equation}
\|\gamma\|\leq C|Z(\gamma)|
\end{equation}
whenever $\gamma\in H_n(X,\ZZ)$ is the class of a compact special Lagrangian submanifold.
\end{prop}

\begin{proof}
Let $\alpha_1,\ldots,\alpha_m$ be $n$-forms on $X$ such that $[\alpha_1],\ldots,[\alpha_m]$ is a basis of $H^n_{dR}(X;\RR)$.
Consider the norm
\begin{equation}
\|\gamma\|=\sum_{i=1}^m\left|[\alpha_i](\gamma)\right|
\end{equation}
on $H_n(X;\RR)$.
Since $X$ is compact and $\Omega$ is non-vanishing on Lagrangian subspaces, there are $C_i>0$ such that $|\alpha_i|_L|\leq C_i|\Omega|_L|$ on any Lagrangian $L\subset T_pX$.
Set $C=\sum C_i$. 
If $L\subset X$ is a compact special Lagrangian submanifold then
\begin{equation}
\|[L]\|=\sum_{i=1}^m\left|\int_L\alpha_i\right|\leq \sum_{i=1}^m\int_L|\alpha_i|_L|\leq C\int_L|\Omega|_L|=C\left|Z(L)\right|
\end{equation}
where the last equality follows since $L$ is special.
\end{proof}

\subsection{Volume and systole}
\label{subsec_volsys}

Let $(M,\omega)$ be a compact symplectic manifold of dimension $2n$ with $\CC$-polarization $\Omega$.
We have two natural top-degree forms on $M$, namely
\begin{equation}
d\mathrm{Vol}_\omega:=\frac{\omega^n}{n!},\qquad d\mathrm{Vol}_\Omega:=(-1)^{\frac{n(n-1)}{2}}\left(\frac{i}{2}\right)^n\Omega\wedge\overline{\Omega}
\end{equation}
which coincide for a Calabi-Yau manifold.
For $n=1,2,3$ we know by the results of Section~\ref{sec_lin_theory} that
\begin{equation}
\frac{d\mathrm{Vol}_\Omega}{d\mathrm{Vol}_\omega}>0\qquad\text{ on }M
\end{equation}
though this could potentially fail for $n>3$.
In any case we can define the \textbf{volume}
\begin{equation}
\mathrm{Vol}_\Omega(M):=\int_Md\mathrm{Vol}_\Omega
\end{equation}
which is positive at least when $n\leq 3$.

Note that $\mathrm{Vol}_\Omega(M)$ depends only on the class of $\Omega$ in $H^n(M;\CC)$, i.e. on the central charge $Z$.
The relevant structure on $H_n(M;\RR)$ is the intersection pairing, which is non-degenerate by Poincar\'e duality.
The intersection pairing corresponds to the Euler pairing (Euler characteristic of $\mathrm{Hom}^\bullet(A,B)$) for the Fukaya category of $M$.
This fact was used by Fan--Kanazawa--Yau~\cite{fky}, see also \cite{fan_systole}, to assign a ``volume'' to a stability condition on a triangulated category. 

Define the \textbf{systole} to be volume of the smallest special Lagrangian submanifold in $M$, i.e.
\begin{equation}
\mathrm{Sys}_\Omega(M):=\inf\left\{\int_L\left|\Omega\right|,L\subset M \text{ compact special Lagrangian}\right\}
\end{equation}
then $\mathrm{Sys}_\Omega(M)>0$ by Proposition~\ref{prop_support}, though possibly $\mathrm{Sys}_\Omega(M)=+\infty$ if no compact special Lagrangian submanifolds exist.
In the case when $(M,\omega,\Omega)$ is Calabi-Yau, this type of systole was considered by Fan~\cite{fan_systole}.
There is a reasonable variant of the definition where one considers all compact Lagrangian submanifolds, not just the special ones. 
We will soon restrict to rational symplectic tori, and since these have a good supply of special Lagrangian subtori, these distinctions will not be important.

\subsection{Second order condition}
\label{subsec_ty}

We will discuss in this subsection a more general version of the notion of a $\CC$-polarization (Definition~\ref{def_Cpol}) suggested by a variant of Hodge theory for symplectic manifolds developed by Tseng--Yau~\cite{tseng_yau}.
To motivate this generalization we consider the famous Kodaira--Thurston nilmanifold example.
Topologically, this is an $S^1$ bundle over $T^3$ which can be constructed as a quotient $M=\RR^4/\ZZ^4$ where
$\ZZ^4$ acts by
\begin{equation}
\left(a,b,c,d\right)\cdot \left(x_1,x_2,x_3,x_4\right):=\left(x_1+a,x_2+b,x_3+c,x_4+d-bx_3\right).
\end{equation}
A frame of the cotangent bundle is given by
\begin{equation}
e_1:=dx_1,\qquad e_2:=dx_2,\qquad e_3:=dx_3,\qquad e_4:=dx_4+x_2dx_3
\end{equation}
in terms of which we can write a symplectic form
\begin{equation}
\omega:=e_1\wedge e_2+e_3\wedge e_4.
\end{equation}
The complex valued form
\begin{equation}
\Omega:=(e_1+ie_2)\wedge (e_3+ie_4)
\end{equation}
is of type $(2,0)$ with respect to a compatible almost complex structure.
However
\begin{equation}
d\Omega=-ie_1\wedge e_2\wedge e_3
\end{equation}
so $\Omega$ is not closed.

To continue we note that $\Omega$ is ``closed'' in a weaker sense:
\begin{equation}\label{ty_closed}
d\Lambda d\Omega=0
\end{equation}
where $\Lambda$ is the dual Lefschetz operator defined by contraction with $\omega$.
This condition ensures that
\begin{equation}
\int_L\Omega
\end{equation}
is invariant under Hamiltonian deformation of $L$.
Furthermore, $\Omega$ gives a class in 
\begin{equation}\label{coho_ph1}
PH_{dd^\Lambda}^n(M):=\frac{\mathrm{Ker}(dd^{\Lambda})\cap P^n(M)}{\mathrm{Im}(d)\cap P^n(M)}.
\end{equation}
defined by Tseng--Yau~\cite{tseng_yau}, where $d^\Lambda=d\Lambda-\Lambda d$ and $P^n(M)$ denotes primitive forms.
The cohomology \eqref{coho_ph1} is finite--dimensional (for compact $M$) and dual to
\begin{equation}\label{coho_ph2}
PH_{d+d^\Lambda}^n(M):=\frac{\mathrm{Ker}(d+d^{\Lambda})\cap P^n(M)}{\mathrm{Im}(dd^\Lambda)\cap P^n(M)}.
\end{equation}
which is the natural home for currents associated with closed Lagrangian submanifolds.

The upshot is that the condition $d\Omega=0$ can be replaced, for present purposes, by the weaker condition~\eqref{ty_closed}.
Imposing the pointwise condition $\Omega_p|_L\neq 0$ for any Lagrangian subspace $L\subset T_pM$ is justified at least in the case of rational symplectic tori, as will be shown in Proposition~\ref{prop_tori_support}.


\section{Rational symplectic tori}

In this section we give an equivalent definition of $\mc U(V)$ as a space of linear functional $\Lambda^n_\mathrm{pr}V\to\CC$ satisfying the support property with respect to homology classes of linear Lagrangian tori in a rational symplectic torus.
In Subsection~\ref{subsec_sysbound} we will establish a systolic bound on $\mc U_\mathrm{ag}(V)$, and in Subsection~\ref{subsec_vol} show that $\mc U(3)/\mathrm{Sp}(6,\ZZ)$ has infinite volume.

A \textbf{rational symplectic torus} is a symplectic torus of the form $(V/\Lambda,\omega)$ where $V$ is an even-dimensional real vector space, $\Lambda\subset V$ a lattice, and $\omega$ a constant symplectic form which takes rational values on $\Lambda\times \Lambda$ (i.e. lies in $H^2(V/\Lambda,\QQ)$).
Any such torus is a product of 2-tori with rational areas, see e.g. \cite[p.~204]{griffiths_harris}.
A symplectic torus which is a product of 2-tori of the same area is a \textbf{standard symplectic torus}.
A \textbf{linear Lagrangian subtorus} of $V/\Lambda$ is a compact Lagrangian torus which is the image of an affine Lagrangian subspace of $V$.
The rationality condition on the symplectic form $\omega$ ensures that $V/\Lambda$ contains many linear Lagrangian tori.
More precisely, the set of Lagrangian subspaces $L\subset V$ which project to compact Lagrangian tori in $V/\Lambda$ is dense in $\mathrm{LGr}(V)$, as follows from rationality of $\mathrm{LGr}(V)$.

To be more concrete, we may assume without loss of generality that $V=\RR^{2n}$, $\omega$ a rational multiple of the standard symplectic form, and 
\begin{equation}
\Lambda=\ZZ^{n}\oplus\ZZ d_1\oplus\ldots\oplus\ZZ d_n
\end{equation}
where $d_1,\ldots,d_n$ are positive integers with $d_1=1$ and $d_k\mid d_{k+1}$.
Any linear Lagrangian subtorus of $\RR^{2n}/\Lambda$ is then given as the image of a block matrix with $n\times n$ blocks $X,Y$ with entries in $\ZZ$ and the $k-th$ row of $Y$ divisible by $d_k$.

\subsection{Support property}

Let $\Gamma$ be a free abelian group of finite rank and $S\subset\Gamma$ a subset, then an additive map $Z:\Gamma\to\CC$ is said to satisfy the \textbf{support property} (with respect to $S$) if there is a norm $\|.\|$ on $\Gamma\otimes\RR$ and a constant $C>0$ such that
\begin{equation}
\|\gamma\|\leq C|Z(\gamma)| \qquad \text{ for all }\gamma\in S.
\end{equation}
Note that if $Z$ satisfies this property then $Z(S)$ is a discrete subset of $\CC$.
The support property was introduced by Kontsevich--Soibelman in the definition of \textit{stability data} and \textit{stability structure} in \cite{ks}, and is usually added to Bridgeland's axioms of a stability condition \cite{bridgeland07}, where $S$ is the set of classes of semistable objects.

\begin{prop}\label{prop_tori_support}
Let $(V/\Lambda,\omega)$ be a rational symplectic torus of dimension $2n$ and $\Omega\in\left(\Lambda^n_\mathrm{pr}V^\vee\right)\otimes\CC$.
Then $\Omega\in\mathcal U(V)$ if and only if $\Omega$, considered as a functional $\Lambda^n_\mathrm{pr}V\to\CC$, satisfies the support property with respect to classes of linear Lagrangian subtori in $V/\Lambda$.
\end{prop}

\begin{proof}
One implication is a special case of Proposition~\ref{prop_support}.
Suppose instead that the support property holds for some $\Omega\in\left(\Lambda^n_\mathrm{pr}V^\vee\right)\otimes\CC$.
Thus there is a $C>0$ such that
\begin{equation} \label{supp_prop_quot}
\frac{1}{C} \leq \frac{|Z(\gamma)|}{\|\gamma\|}
\end{equation} 
for $\gamma\in\Lambda^n_\mathrm{pr}\Lambda=:\Gamma$ the class of a linear Lagrangian subtorus.
The right hand side of \eqref{supp_prop_quot} is a well defined continuous function on $\mathrm{LGr}(V)$, and since it is bounded below on a dense subset, it must be non-vanishing everywhere.
This proves that $\Omega$ does not vanish on any Lagrangian subspace in $V$.
\end{proof}

The previous proposition provides good evidence that
\begin{equation}
\mathrm{Stab}(\mc F(V/\Lambda))=\widetilde{\mc U^+(V)}
\end{equation}
where the left hand side is the space of stability conditions on the Fukaya category (over the Novikov field) of a rational symplectic torus, and the right hand side is the universal cover of the $K(\ZZ,1)$-space $\mc U^+(V)$.
A similar conjecture was suggested by Kontsevich~\cite{kontsevich_lille}.

\subsection{Systolic bound}
\label{subsec_sysbound}

Let $(M,\omega)$ be a compact symplectic manifold.
A \textbf{systolic bound} is an inequality of the form
\begin{equation}\label{sys_bound}
\left(\mathrm{Sys}_{\Omega}(M)\right)^2\leq C\mathrm{Vol}_\Omega(M)
\end{equation}
where the constant $C>0$ depends only on $(M,\omega)$ and not on $\Omega$, which is allowed to vary in some set of $\CC$-polarizations (to be specified).
Intuitively it means that if the volume is fixed, then the smallest special Lagrangian cannot be too big.
Thus, in order to prove a systolic bound one needs to show existence of a sufficiently small special Lagrangian.

\begin{theorem}\label{thm_torus_sys}
Let $(V/\Lambda,\omega)$ be a rational symplectic torus, then a systolic bound holds for $\Omega\in\mc U_\mathrm{geom}(V)$.
\end{theorem}

The proof will be based on Minkowski's lattice point theorem, which states that if $\Lambda\subset\RR^n$ is a lattice, then
\begin{equation}
\min_{x\in\Lambda\setminus\{0\}} \|x\| \leq C_n\sqrt[n]{\mathrm{Vol}(\RR^n/\Lambda)}
\end{equation}
for some constant $C_n>0$ depending only on $n$, not $\Lambda$.
In fact, the $n=1$ case of Theorem~\ref{thm_torus_sys} is just the $n=2$ case of the lattice point theorem.

\begin{proof}
It suffices to consider the standard symplectic torus $\RR^{2n}/\ZZ^{2n}$, since passing to a finite cover can only scale systole and volume by a fixed finite amount.

We will make use of Siegel's description of the space of compatible complex structures on $V$, see~\cite{siegel39}.
Recall that the Siegel space is the set $\mathcal S_n$ of matrices $Z=X+iY\in\mathrm{Mat}(n\times n,\CC)$ with $X,Y$ symmetric and $Y$ positive definite.
The group $\mathrm{Sp}(2n,\RR)$ acts transitively on $\mathcal S_n$ via
\begin{equation}
\begin{pmatrix} A & B \\ C & D \end{pmatrix}\cdot Z:=(AZ+B)(CZ+D)^{-1}.
\end{equation}
The point $X+iY\in\mathcal S_n$ corresponds to the K\"ahler torus $\RR^{2n}/\Lambda$ with
\begin{equation}\label{siegel_torus}
\Lambda=\begin{pmatrix} 1 & X \\ 0 & Y \end{pmatrix}\ZZ^{2n},\qquad
J=\begin{pmatrix} 0 & -1 \\ 1 & 0 \end{pmatrix},\qquad
\omega=\begin{pmatrix} 0 & Y^{-1} \\ -Y^{-1} & 0 \end{pmatrix}
\end{equation}
which is isomorphic to the standard torus $\RR^{2n}/\ZZ^{2n}$ with complex structure
\begin{equation}
J'=\begin{pmatrix} -XY^{-1} & -Y-XY^{-1}X \\ Y^{-1} & Y^{-1}X \end{pmatrix}.
\end{equation}

Fix a compatible complex structure on $V$, i.e. a point $Z=X+iY\in\mathcal S_n$, and
let $\Omega=dz_1\wedge\ldots\wedge dz_n$, without loss of generality (since the inequality \eqref{sys_bound} is homogeneous in $\Omega$).
For very general reasons (Proposition~\ref{prop_support}) there must a linear Lagrangian subtorus in $\RR^{2n}/\Lambda$ with minimal volume among such.
Applying some element of $\mathrm{Sp}(2n,\ZZ)$ we may assume that this subtorus is $(\RR^n/\ZZ^n)\times\{0\}\subset \RR^{2n}/\Lambda$.
More explicitly, this means that
\begin{equation}\label{sys_condition_mat}
|\det(M+ZN)|\geq 1
\end{equation}
for any pair of matrices $M,N\in\mathrm{Mat}(n\times n,\ZZ)$ such that $M^TN=N^TM$ and the rank of the block matrix with blocks $M$ and $N$ on top of each other is $n$.

Applying Minkowski's theorem to the lattice $\ZZ^n$ with the positive definite quadratic form $Y$, we can assume, after performing some change of coordinates given by a block--diagonal element of $\mathrm{Sp}(2n,\ZZ)$, that
\begin{equation}\label{Y_minkowski_bound}
y_{11}\leq C_n\sqrt[n]{\det(Y)}
\end{equation}
where $Y=(y_{ij})_{1\leq i,j\leq n}$.

Given $p,q\in\ZZ$, not both zero, let
\begin{equation}
M=\mathrm{diag}(p,1\ldots,1),\qquad N=\mathrm{diag}(q,0,\ldots,0)
\end{equation}
then $|\det(M+ZN)|=|p+z_{11}q|\geq 1$ by \eqref{sys_condition_mat}. 
Applying Minkowski's theorem to $\CC$ with the standard metric and the lattice $\ZZ\oplus z_{11}\ZZ $ we find
\begin{equation}
y_{11}\geq\frac{1}{C_1}.
\end{equation}
Combining the above we get
\begin{equation}
\left(\mathrm{Sys}_\Omega(\RR^{2n}/\Lambda)\right)^2 \leq 1 \leq C\det(Y)=C\mathrm{Vol}_\Omega(\RR^{2n}/\Lambda)
\end{equation}
for $C:=(C_1C_n)^n$.
\end{proof}

The previous theorem can be strengthened to allow also those $\Omega\in\mc U^+(V)$ which are in the $\mathrm{GL}^+(2,\RR)$--orbit of an $(n,0)$ form.

\begin{theorem}\label{thm_torus_sys_ag}
Let $(V/\Lambda,\omega)$ be a rational symplectic torus, then a systolic bound holds for $\Omega\in\mc U_\mathrm{ag}(V)$.
\end{theorem}

\begin{proof}
We will deal with the cases of odd and even $n$ separately, where $2n=\dim V$. 
First suppose that $n$ is even.
Any $\Omega'$ which is in the $\mathrm{GL}^+(2,\RR)$ orbit of an element in $\mc U_\mathrm{geom}(V)$ can be written in the form
\begin{equation}
\Omega'=\Omega+\epsilon\overline{\Omega}, \qquad \epsilon\in\CC, |\epsilon|<1,\Omega\in\mc U^+_\mathrm{geom}(V).
\end{equation}
Then we have
\begin{equation}
\mathrm{Sys}_{\Omega'}(T^{2n})\leq 2\mathrm{Sys}_\Omega(T^{2n})
\end{equation}
and
\begin{equation}
\Omega'\wedge\overline{\Omega'}=\left(1+|\epsilon|^2\right)\Omega\wedge\overline{\Omega}
\end{equation}
since $n$ is even, thus $\mathrm{Vol}_{\Omega'}(T^{2n})\geq \mathrm{Vol}_\Omega(T^{2n})$.
Combining this with Theorem~\ref{thm_torus_sys} gives
\begin{equation}
\left(\mathrm{Sys}_{\Omega'}(T^{2n})\right)^2\leq 4C\mathrm{Vol}_{\Omega'}(T^{2n})
\end{equation}
for any $\Omega'\in\mc U_\mathrm{ag}(V)$.

Assume now that $n$ is odd, so that $\mathrm{Vol}_{\Omega}(T^{2n})$ is invariant under the $\mathrm{SL}(2,\RR)$--action.
The proof of Theorem~\ref{thm_torus_sys} showed that there is a lattice $\ZZ\oplus\tau\ZZ\subset \CC$ (where $\tau:=z_{11}$) all of whose non-zero elements are central charges of (possibly immersed) linear Lagrangian tori in $T^{2n}$, and
furthermore the area of this lattice, $\mathrm{Im}(\tau)$, has an absolute upper bound for fixed $\mathrm{Vol}_{\Omega}(T^{2n})$.
This property persists when applying some element of $\mathrm{SL}(2,\RR)$ to $\Omega$, so again by Minkowski's theorem in the plane we get an upper bound on the smallest volume of a linear Lagrangian subtorus. 
\end{proof}

As a corollary we find that a systolic bound holds on all of $\mc U^+(2)$.

\begin{q}
Does a systolic bound hold on $\mc U^+(3)$?
\end{q}

\subsection{Volume of moduli space}
\label{subsec_vol}

Suppose $V$ is a symplectic vector space of dimension $2n$ with $n\geq 1$ odd, then $\left(\Lambda^n_{\mathrm{pr}}V^\vee\right)\otimes \CC$ has a natural symplectic structure
\begin{equation}
(\Omega_1,\Omega_2)\mapsto\mathrm{Im}\left(\Omega_1\wedge\Omega_2\right)\left(\frac{\omega^n}{n!}\right)^{-1}
\end{equation}
which is invariant under the action of $\mathrm{Sp}(V)$ and the subgroup 
\begin{equation}\label{TeichFlow_subgroup}
\left\{\begin{pmatrix} e^t & 0 \\ 0 & e^{-t} \end{pmatrix} \mid t\in\RR \right\}\subset\mathrm{SL}(2,\RR).
\end{equation}
The volume form associated with the symplectic structure is invariant under the full group $\mathrm{SL}(2,\RR)$.
By restriction, the same applies to $\mc U^+(V)$, which contains the hypersurface
\begin{equation}
\mc U_1^+(V):=\left\{\Omega\in\mc U^+(\RR^{2n})\mid \frac{\omega^n}{n!}=(-1)^{\frac{n(n-1)}{2}}\left(\frac{i}{2}\right)^n\Omega\wedge\overline{\Omega}\right\}
\end{equation}
of contact type. (The Euler vector field serves as a Liouville vector field.)
The actions of $\mathrm{Sp}(V)$ and $\mathrm{SL}(2,\mathbb R)$ restrict to $\mc U_1^+(V)$.

Let $n\geq 1$ be odd as before and let $T^{2n}=\RR^{2n}/\ZZ^{2n}$ be the standard symplectic torus and consider the orbifold
\begin{equation}
\mathcal M_1(n):=\mc U_1^+(\RR^{2n})/\mathrm{Sp}(2n,\ZZ).
\end{equation}
which is the  hypersurface in $\mc U^+(\RR^{2n})/\mathrm{Sp}(2n,\ZZ)$ where $\mathrm{Vol}_\Omega(T^{2n})=1$.
By the above discussion we have a natural $\mathrm{SL}(2,\RR)$--invariant measure on $\mc M_1(n)$.
Note that $\mc M_1(1)=\mathrm{SL}(2,\RR)/\mathrm{SL}(2,\ZZ)$ is the moduli space of flat structures on the 2-torus with normalized area and preferred choice of horizontal direction, and 
\begin{equation}
\mathrm{Vol}\left(\mc M_1(1)\right)<\infty
\end{equation}
which is a very special case of a theorem due to Masur~\cite{masur82} and Veech~\cite{veech82} on the finiteness of the volumes of moduli spaces of flat surfaces.

In contrast, the volume turns out to be infinite for $n=3$.

\begin{theorem}\label{thm_vol_inf}
$\mathrm{Vol}\left(\mc M_1(3)\right)=\infty$
\end{theorem}

In fact the proof will show a stronger result, that the total volume must be infinite for any non-vanishing Borel measure which is invariant under the action of the subgroup  \eqref{TeichFlow_subgroup}.

\begin{proof}
Let $V$ be a 6-dimensional symplectic vector space.
Given $\alpha\in\Lambda^3V^\vee$ we defined in \eqref{ThreeFormQ} a quadratic form $q_\alpha$ on $V$ which depends quadratically on $\alpha$, i.e. $q_{\lambda\alpha}=\lambda^2q_\alpha$.
Furthermore, if $\Omega\in\mc U(V)$ then $q_{\mathrm{Re}\Omega}$ is positive definite by Lemma~\ref{lem_Qposdef}.
To get a single number from $q_\alpha$ we first define $K_\alpha:V\to V$ by
\begin{equation}
\omega(K_\alpha v,w)=q_\alpha(v,w)
\end{equation}
and let
\begin{equation}
d_\alpha:=\det(K_\alpha)
\end{equation}
then $d_{\lambda \alpha}=\lambda^{12}d_\alpha$ and $d_{\mathrm{Re}\Omega}>0$ for $\Omega\in\mc U(V)$.
Thus we can define a function 
\begin{equation}
f:\mc M_1(3)\to\RR,\qquad f(\Omega)=\log \left|d_{\lambda \alpha}\right|.
\end{equation}

Let $T:\mc M_1(3)\to\mc M_1(3)$ be the action of $\mathrm{diag}(2,1/2)\in\mathrm{SL}(2,\RR)$, then
\begin{align}
f(T\cdot\Omega)&=\log |d_{\mathrm{Re}(T\cdot\Omega)}| \\
               &=\log |d_{2\mathrm{Re}(\Omega)}| \\
               &=f(\Omega)+C
\end{align}
where $C:=12\log 2$.
This implies that each subset 
\begin{equation}
f^{-1}([kC,(k+1)C))\subset \mc M_1(3),\qquad k\in\ZZ
\end{equation}
has the same positive volume.
Since the whole space is the disjoint union of these subsets, the claim follows.
\end{proof}

\bibliographystyle{plain}
\bibliography{geomstab}

\end{document}